\documentclass[a4paper,12pt,final]{amsart}
\usepackage{times,a4wide,mathrsfs,amssymb,amsmath,amsthm,wasysym}

\newcommand{\C}{\mathbb{C}}
\newcommand{\ZZ}{\mathbb{Z}}

\newcommand{\QQ}{\mathbb{Q}}
\newcommand{\NN}{\mathbb{N}}
\newcommand{\PP}{\mathbb{P}}

\newcommand{\OO}{\mathcal O}
\newcommand{\Ss}{\mathcal S}

\newcommand{\MM}{\mathcal M}

\newcommand{\pic}{\hbox{Pic}}

\newcommand{\rom}{\romannumeral}

 \makeatletter
\newcommand*{\da@rightarrow}{\mathchar"0\hexnumber@\symAMSa 4B }
\newcommand*{\da@leftarrow}{\mathchar"0\hexnumber@\symAMSa 4C }
\newcommand*{\xdashrightarrow}[2][]{%
  \mathrel{%
    \mathpalette{\da@xarrow{#1}{#2}{}\da@rightarrow{\,}{}}{}%
  }%
}
\newcommand{\xdashleftarrow}[2][]{%
  \mathrel{%
    \mathpalette{\da@xarrow{#1}{#2}\da@leftarrow{}{}{\,}}{}%
  }%
}
\newcommand*{\da@xarrow}[7]{%
  \sbox0{$\ifx#7\scriptstyle\scriptscriptstyle\else\scriptstyle\fi#5#1#6\m@th$}%
  \sbox2{$\ifx#7\scriptstyle\scriptscriptstyle\else\scriptstyle\fi#5#2#6\m@th$}%
  \sbox4{$#7\dabar@\m@th$}%
  \dimen@=\wd0 %
  \ifdim\wd2 >\dimen@
    \dimen@=\wd2 %
  \fi
  \count@=2 %
  \def\da@bars{\dabar@\dabar@}%
  \@whiledim\count@\wd4<\dimen@\do{%
    \advance\count@\@ne
    \expandafter\def\expandafter\da@bars\expandafter{%
      \da@bars
      \dabar@ 
    }%
  }%
  \mathrel{#3}%
  \mathrel{%
    \mathop{\da@bars}\limits
    \ifx\\#1\\%
    \else
      _{\copy0}%
    \fi
    \ifx\\#2\\%
    \else
      ^{\copy2}%
    \fi
  }%
  \mathrel{#4}%
}
\makeatother

\DeclareMathOperator{\aut}{Aut}

\DeclareMathOperator{\ide}{id}

\DeclareMathOperator{\ima}{Im}

\newtheorem{theorem}{Theorem}[section]
\newtheorem{claim}[theorem]{Claim}
\newtheorem{lemma}[theorem]{Lemma}

\newtheorem{corollary}[theorem]{Corollary}
\newtheorem{proposition}[theorem]{Proposition}

\newtheorem{remark}[theorem]{Remark}
\newtheorem{definition}[theorem]{Definition}
\newtheorem{convention}{Conventions}

\newtheorem{notation}[theorem]{Notation}

\newtheorem{nonumbering}{Theorem}

\newtheorem{nonumberingt}{Acknowledgements}

\begin{document}
\author[Robert Laterveer]
{Robert Laterveer}

\address{Institut de Recherche Math\'ematique Avanc\'ee,
CNRS -- Universit\'e 
de Strasbourg,\
7 Rue Ren\'e Des\-car\-tes, 67084 Strasbourg CEDEX,
FRANCE.}
\email{robert.laterveer@math.unistra.fr}

\title{Algebraic cycles and special Horikawa surfaces}

\begin{abstract} This note is about a certain $16$-dimensional family of surfaces of general type with $p_g=2$ and $q=0$ and $K^2=1$, called ``special Horikawa surfaces''. These surfaces, studied by Pearlstein--Zhang and by Garbagnati, are related to K3 surfaces. We show that special Horikawa surfaces have a multiplicative Chow--K\"unneth decomposition, in the sense of Shen--Vial. As a consequence, the Chow ring of special Horikawa surfaces displays K3-like behaviour.
 \end{abstract}


\maketitle

\section{Introduction}

Horikawa surfaces are minimal complex surfaces of general type verifying either
  \[ K^2=2p_g-4\ \ \ \hbox{and\ $K^2$\ is\ even}\ ,  \]
  or 
   \[ K^2=2p_g-3\ \ \ \hbox{and\ $K^2$\ is\ odd}\  \]
   (i.e., Horikawa surfaces lie on, or immediately below, the Noether line) \cite{BPV}, \cite{Hor}.
   
  Pearlstein--Zhang \cite{PZ} and Garbagnati \cite{Gar} have studied so-called {\em special Horikawa surfaces\/}; by definition, these are Horikawa surfaces with $K^2=1$ and $p_g=2$ obtained as
  bidouble covers of $\PP^2$ branched along a quintic and two lines.
  From the cohomological viewpoint, a special Horikawa surface $S$ looks like a ``K3 Big Mac''. That is, there exist two K3 surfaces $X_1, X_2$ and an isomorphism
    \begin{equation}\label{incoh} H^2_{tr}(S,\QQ)\cong H^2_{tr}(X_1,\QQ)\oplus H^2_{tr}(X_2,\QQ)\ .\end{equation} 
    (Here, the transcendental cohomology $H^2_{tr}()$ is defined as the orthogonal complement of the N\'eron--Severi group with respect to the cup product.)
    
    The Bloch--Beilinson--Murre conjectures \cite{J2}, \cite{Vo}, \cite{MNP}, \cite{Mur} make the oracular prediction that the Big Mac relation (\ref{incoh}) should also hold on the level of the Chow group of $0$-cycles. The first (easy) result of this note confirms that this is indeed the case:
    
    \begin{nonumbering}[=Theorem \ref{main}] Let $S$ be a special Horikawa surface, and let $X_1, X_2$ be the associated K3 surfaces. There is an isomorphism
     \[ A^2_{hom}(S)\cong A^2_{hom}(X_1)\oplus A^2_{hom}(X_2)\ .\]
     \end{nonumbering}
     
     (Here, $A^2_{hom}()$ denotes the Chow group of degree $0$ $0$-cycles with $\QQ$-coefficients modulo rational equivalence.)
     
      The relation of Theorem \ref{main} also holds on the level of Chow motives. This gives some new examples of surfaces of general type with 
      finite-dimensional motive, in the sense of Kimura \cite{Kim} (cf. Corollary \ref{cor1}).
     
%
      
The second (more interesting) result of this note concerns the ring structure of the Chow ring, given by intersection product.
We show that for special Horikawa surfaces, the Chow ring behaves just like that of K3 surfaces:
   
   \begin{nonumbering}[=Theorem \ref{th:mck}] Let $S$ be a special Horikawa surface. Then $S$ has a multiplicative Chow--K\"unneth decomposition (in the sense of Shen--Vial \cite{SV}). In particular, all intersections of divisors are proportional in the Chow group of 0-cycles:
     \[   \ima \bigl( A^1_{}(S)\otimes A^1_{}(S)\ \to\ A^2(S)\bigr) \stackrel{}{=}\QQ[c_2(S)]\ .\]
     (Here $c_2(S)$ denotes the Chow-theoretic second Chern class of the tangent bundle.)
    \end{nonumbering}
    
 This result can be seen as part of a general program aimed at understanding which varieties admit a multiplicative Chow--K\"unneth decomposition (for more on this program, cf. \S \ref{ssmck} and the references given there).

 \vskip0.6cm

\begin{convention} In this note, the word {\sl variety\/} will refer to a reduced irreducible scheme of finite type over $\C$. A {\sl subvariety\/} is a (possibly reducible) reduced subscheme which is equidimensional. 

{\bf All Chow groups will be with rational coefficients}: we will denote by $A_j(X)$ the Chow group of $j$-dimensional cycles on $X$ with $\QQ$-coefficients; for $X$ smooth of dimension $n$ the notations $A_j(X)$ and $A^{n-j}(X)$ are used interchangeably. 

The notations $A^j_{hom}(X)$, $A^j_{AJ}(X)$ will be used to indicate the subgroups of homologically trivial, resp. Abel--Jacobi trivial cycles.
For a morphism $f\colon X\to Y$, we will write $\Gamma_f\in A_\ast(X\times Y)$ for the graph of $f$.
The contravariant category of Chow motives (i.e., the category where Hom-groups are defined using Chow groups with rational coefficients as in \cite{Sc}, \cite{MNP}) will be denoted $\MM_{\rm rat}$.


\end{convention}

\section{Preliminaries}

\subsection{Special Horikawa surfaces}

\begin{proposition}[Pearlstein--Zhang \cite{PZ}]\label{prop1}
 Let $C\subset\PP^2$ be a smooth quintic curve, and let $L_1, L_2$ be distinct lines intersecting $C$ transversely and such that $C\cap L_1\cap L_2=\emptyset$.
There exists a surface $S$ obtained as the bidouble cover of $\PP^2$ branched along $C+L_1+L_2$. There exist morphisms
  \[ \bar{f}_j\colon\ \ \ S\ \to\ \bar{X}_j\ \ \ (j=1,2)\ ,\]
  where $\bar{X}_j$ is a double cover of $\PP^2$ branched along $C+L_j$. The surface $S$ is a minimal surface of general type with $p_g(S)=2$ and $K_S^2=1$. Moreover, $S$ is simply-connected and $K_S$ is ample.
  \end{proposition}
  
  \begin{definition}[Pearlstein--Zhang \cite{PZ}] A surface $S$ as in Proposition \ref{prop1} will be called a {\em special Horikawa surface\/}. The K3 surfaces $X_1, X_2$ obtained by resolving the singularities of the double covers $\bar{X}_1, \bar{X}_2$ will be called the {\em associated K3 surfaces}.
  \end{definition}

  \begin{remark} The construction on which Proposition \ref{prop1} is based also occurs in recent work of Garbagnati \cite[Section 5.4.1]{Gar}; the special Horikawa surfaces are exactly the surfaces $S_6^{(1)}$ of \cite[Proposition 5.20]{Gar}.
  
  The focus of the two papers \cite{Gar} and \cite{PZ} is quite different: in \cite{Gar}, special Horikawa surfaces occur as examples in the classification of smooth double covers of K3 surfaces; in \cite{PZ}, on the other hand, the main result is a Torelli theorem for special Horikawa surfaces.
  \end{remark}

\begin{proposition}[\cite{PZ}, \cite{Gar}]\label{Htr} Let $S$ be a special Horikawa surface, and let $X_1, X_2$ be the associated K3 surfaces. There is an isomorphism of Hodge structures
  \[ H^2_{tr}(S,\QQ)\cong H^2_{tr}(X_1,\QQ)\oplus H^2_{tr}(X_2,\QQ)\ .\]
  (Here for any surface $Y$, $H^2_{tr}(Y)$ denotes the {\em transcendental cohomology\/}, i.e. the orthogonal of $NS(Y)\subset H^2(Y)$ with respect to the cup product.)
  \end{proposition}
  
 \begin{proof} This is established in \cite[5.4.1]{Gar}, and also in \cite[Remark 1.10]{PZ}. For later use, we briefly resume the argument. Let $\sigma_1, \sigma_2$ be the involutions of $S$ such that $S/\sigma_j$ is the ``singular K3 surface'' $\bar{X}_j$. Then the cohomology of $S$ decomposes
  \begin{equation}\label{H2decomp}
     H^2(S,\QQ) =H^2(S,\QQ)^{+,+}\oplus H^2(S,\QQ)^{+,-}\oplus H^2(S,\QQ)^{-,+}\oplus H^2(S,\QQ)^{-,-}\ ,
     \end{equation}
   where $ H^2(S,\QQ)^{\pm,\mp}$ is the subspace where $\sigma_1$ acts as $\pm$ the identity, and  $\sigma_2$ acts as $\mp$ the identity.
     
    The first summand of (\ref{H2decomp}) corresponds to $H^2(\PP^2,\QQ)$, and so it is contained in $NS(S)_\QQ$. Likewise, the last summand corresponds to $H^2(W,\QQ)$, where $W$ is the double cover of $\PP^2$ branched along $L_1\cup L_2$. Since $W$ is rational, the last summand is also contained in $NS(S)_{\QQ}$.
     It follows that (\ref{H2decomp}) induces a decomposition
   \[    H^2_{tr}(S,\QQ) = H^2_{tr}(S,\QQ)^{+,-}\oplus H^2_{tr}(S,\QQ)^{-,+} \ .\]
   Here, the first summand is equal to the subspace of $H^2_{tr}(S,\QQ)$ where $(\sigma_1)^\ast=\ide$ (since, as we have just seen, $H^2_{tr}(S,\QQ)^{+,+}=0$). This means that
    \[    H^2_{tr}(S,\QQ)^{+,-} = (\bar{p}_1)^\ast H^2_{tr}(\bar{X}_1,\QQ)\ ,\]
    where $\bar{f}_1\colon S\to S/\sigma_1=:\bar{X}_1$ is the quotient morphism. Since the resolution morphism $g_1\colon X_1\to \bar{X}_1$ induces an isomorphism on $H^2_{tr}()$, it follows that
    \[     H^2_{tr}(S,\QQ)^{+,-} = (g_1)_\ast (\bar{f}_1)^\ast H^2_{tr}(\bar{X}_1,\QQ)\ .\]
   
   The set-up being symmetric with respect to $X_1, X_2$, we likewise find an isomorphism
     \[     H^2_{tr}(S,\QQ)^{-,+} = (g_2)_\ast (\bar{f}_2)^\ast H^2_{tr}(\bar{X}_2,\QQ)\ .\]  
     This proves the proposition.   
 \end{proof} 
  
\begin{proposition}[\cite{PZ}]\label{pz} Let $S$ be a special Horikawa surface. The surface $S$ is isomorphic to a smooth hypersurface 
  \[  x_3^2 =g(x_0^2,x_1^2,x_2) \]
  in weighted projective space $\PP(1,1,2,5)$, where $g$ is homogeneous of degree $5$ in $x_0^2,x_1^2,x_2$.
  Conversely, a smooth hypersurface in $\PP(1,1,2,5)$ of this type defines a special Horikawa surface.  
\end{proposition}
  
 \begin{proof} This is \cite[Proposition 1.6]{PZ}.
 \end{proof}

 \begin{remark} As noted in \cite{PZ}, there is an analogy with Kunev surfaces. 
 
 As shown by Catanese \cite{Cat}, the canonical model of a surface of general type with $K^2=p_g=1$
 is a complete intersection in weighted projective space $\PP(1,2,2,3,3)$. These canonical models form an $18$-dimensional family. Inside this family, there is a $12$-dimensional subfamily for which the bicanonical map is a bidouble cover of $\PP^2$; these surfaces are called ``special'' in \cite{Cat}, and they are also called ``Kunev surfaces'' or ``Todorov surfaces with $K^2=1$'' \cite{Kun}, \cite{Tod}.
 
Similarly, the canonical model of any surface of general type with $K^2=1$, $p_g=2$ and $q=0$ can be described as a hypersurface in $\PP(1,1,2,5)$ \cite{Hor}, \cite[VII.7]{BPV}. Inside this family, the special Horikawa surfaces form a $16$-dimensional subfamily where the bicanonical map is a bidouble cover of $\PP^2$ branched along a quintic and two lines.
  \end{remark} 
  
  \subsection{Quotient varieties}
\label{ssq}

\begin{definition} A {\em projective quotient variety\/} is a variety
  \[ X=Y/G\ ,\]
  where $Y$ is a smooth projective variety and $G\subset\hbox{Aut}(Y)$ is a finite group.
  \end{definition}
  
 \begin{proposition}[Fulton \cite{F}]\label{quot} Let $X$ be a projective quotient variety of dimension $n$. Let $A^\ast(X)$ denote the operational Chow cohomology ring. The natural map
   \[ A^i(X)\ \to\ A_{n-i}(X) \]
   is an isomorphism for all $i$.
   \end{proposition}
   
   \begin{proof} This is \cite[Example 17.4.10]{F}.
      \end{proof}

\begin{remark} It follows from Proposition \ref{quot} that the formalism of correspondences goes through unchanged for projective quotient varieties (this is also noted in \cite[Example 16.1.13]{F}). We can thus consider motives $(X,p,0)\in\MM_{\rm rat}$, where $X$ is a projective quotient variety and $p\in A^n(X\times X)$ is a projector. For a projective quotient variety $X=Y/G$, one readily proves (using Manin's identity principle) that there is an isomorphism
  \[  h(X)\cong h(Y)^G:=(Y,\Delta^G_Y,0)\ \ \ \hbox{in}\ \MM_{\rm rat}\ ,\]
  where $\Delta^G_Y$ denotes the idempotent ${1\over \vert G\vert}{\sum_{g\in G}}\Gamma_g$.  
  \end{remark}

\subsection{Multiplicative Chow--K\"unneth decomposition}
\label{ssmck}
  
The notion of multiplicative Chow--K\"unneth decomposition was introduced by Shen--Vial \cite[\S
	8]{SV}.  This notion provides an explicit candidate for Beauville's
	conjectural splitting of the conjectural Bloch--Beilinson filtration on the Chow
	rings of hyperk\"ahler varieties.
	
	First, let us recall the notion of Chow--K\"unneth decomposition:
	
	\begin{definition}[Murre \cite{Mur}] 
		Let $X$ be a smooth projective variety of
		dimension $n$. We say that $X$ has a {\em Chow--K\"unneth decomposition\/} (CK
		decomposition for short) if there exists a decomposition of the diagonal
		\[ \Delta_X= \pi^0_X+ \pi^1_X+\cdots +\pi^{2n}_X\ \ \ \hbox{in}\ A^n(X\times
		X)\ ,\]
		such that the $\pi^i_X$ are mutually orthogonal idempotents and
		$(\pi^i_X)_\ast H^\ast(X)= H^i(X)$.
	\end{definition}
	
	Assuming the Bloch--Beilinson conjectures, Jannsen \cite{J2} proved that all
	smooth projective varieties admit a CK decomposition, and moreover the CK
	projectors $\pi^i_X$ induce a splitting of the Bloch--Beilinson filtration on
	the Chow {groups}. A sufficient condition for the induced splitting to be
	compatible with the ring structure is given by the following definition:
	
	\begin{definition}[Shen--Vial \cite{SV}] 
		Let $X$ be a smooth projective variety
		of dimension $n$. Let $\Delta^{sm}_X \in A^{2n}(X\times X\times X)$ be the class
		of
		the small diagonal
				\[ \Delta^{sm}_X:=\bigl\{ (x,x,x)\ \vert\ x\in X\bigr\}\ \subset\ X\times X\times
		X\ .\]
		A \emph{multiplicative Chow--K\"unneth decomposition} (MCK decomposition for
		short) is a CK decomposition $\{\pi^i_X\}$ of $X$ that is {\em
			multiplicative\/}, \emph{i.e.} that satisfies
		\[ \pi^k_X\circ \Delta^{sm}_X \circ (\pi^i_X\times \pi^j_X)=0\ \ \ \hbox{in}\
		A^{2n}(X\times X\times X)\ \ \ \hbox{for\ all\ }i+j\not=k\ .\]
		An MCK decomposition is necessarily {\em self-dual\/}, \emph{i.e.} it
		satisfies
		\[ \pi^k_X={}^t \pi^{2n-k}_X\ \ \  \forall \ k\ \]
		(where the superscript ${}^t ()$ indicates
		the transpose correspondence),
		cf. \cite[footnote 24]{FV}.
	\end{definition}

	\begin{remark}
	It follows from the definition that if $X$ has an MCK decomposition $\{\pi^i_X\}$, then setting
	\[ A^i_{(j)}(X):= (\pi_X^{2i-j})_\ast A^i(X) \ ,\]
	one obtains a bigraded ring structure on the Chow ring: that is, the
	intersection product sends 
	$A^i_{(j)}(X)\otimes A^{i^\prime}_{(j^\prime)}(X) $ to 
	$A^{i+i^\prime}_{(j+j^\prime)}(X)$. 
	
	While a CK decomposition is expected to exist for any smooth projective variety \cite{Mur}, \cite{J2}, the property of having
	an MCK decomposition is restrictive. For example, a very general
	curve of genus $\geq 3$ does not admit an MCK decomposition. The existence of
	an MCK decomposition is closely related to Beauville's ``weak splitting property''
	\cite{Beau3}, and it is conjectured that
	hyperk\"ahler varieties admit an MCK decomposition \cite[Conjecture~4]{SV}. 
	For a K3 surface $S$, the seminal work of Beauville--Voisin \cite{BV} establishes 
	the existence of a canonical zero-cycle $o_S\in A^2(S)$ of degree 1 that
	``decomposes'' the small diagonal in $S\times S\times S$.
	As observed in \cite[Proposition~8.14]{SV}, this can be interpreted as saying that the
	CK decomposition defined by 
	  \[ \pi_S^0 := o_S\times S\ , \ \pi_S^4 :=
	        S\times o_S\ ,\ \pi_S^2 = \Delta_S - \pi_S^0 - \pi_S^4 \] 
	is multiplicative.
	The MCK conjecture for hyperk\"ahler varieties
	has been established for Hilbert
	schemes of length $n$ subschemes on K3 surfaces \cite{V6}, \cite{NOY}, and for
	generalized Kummer varieties \cite{FTV}. Other examples of varieties admitting an MCK decomposition
	can be found in \cite{SV2}, \cite{LV}, \cite{FLV}, \cite{d3}, \cite{Ver}, \cite{S2}, \cite{B1B2}. 	
	\end{remark}
	 
  For later use, we record the following useful equivalence:
  
  \begin{proposition}[Shen--Vial \cite{SV}]\label{equiv} Let $X$ be a smooth projective surface with $H^1(X,\QQ)=0$. The following are equivalent:
  
  \noindent
  (\rom1) $X$ has an MCK decomposition;
  
  \noindent
  (\rom2) there exists $o_X\in A^2(X)$ such that 
     \[\begin{split} \Gamma_{3}(X,
			o_X):=&\Delta^{sm}_{X}-p_{12}^{*}(\Delta_{X})p_{3}^{*}(o_X)-p_{23}^{*}(\Delta_{X})p_{1}^{*}(o_X)-p_{13}^{*}(\Delta_{X})p_{2}^{*}(o_X)\\
			&+p_{1}^{*}(o_X)p_{2}^{*}(o_X)+p_{1}^{*}(o_X)p_{3}^{*}(o_X)+p_{2}^{*}(o_X)p_{3}^{*}(o_X) \ \  =0\ \ \ \ \ \hbox{in}\ A^4(X^3)\ .\\
			\end{split}\]  
  (the cycle $\Gamma_3(X,o_X)$ is known as the {\em modified small diagonal\/}). 
   \end{proposition}
   
   \begin{proof} This is \cite[Proposition 8.14]{SV}.
     \end{proof}

%

  \subsection{Relative K\"unneth projectors}
  
   \begin{notation}\label{fam} Let
   \[ \Ss\ \to\ B \]
   denote the family of all smooth hypersurfaces in $\PP:=\PP(1,1,2,5)$ of type
   \[ f_b(x_0,x_1,x_2,x_3)=0\ ,\]
   where $f_b$ is weighted homogeneous of degree $10$, and $x_0, x_1$ occur only in even degree. Let $S_b$ denote the fibre of $\Ss$ over $b\in B$.
   
   Let $\sigma_j\in\aut(\PP)$, $j=1,2$, be the involutions defined as
   \[  \begin{split}  \sigma_1 [x_0:x_1:x_2:x_3] =[-x_0:x_1:x_2:x_3]\ ,\\
                         \sigma_2  [x_0:x_1:x_2:x_3] =[x_0:-x_1:x_2:x_3]\ \\
                         \end{split}\]
%
    \end{notation}
    
   \begin{lemma} Let $\Ss\to B$ be as in Notation \ref{fam}. For any $b\in B$, the fibre $S_b$ is isomorphic to a special Horikawa surface. The quotients
   $S_b/\langle \sigma_j\vert_b\rangle$ are the ``singular K3 surfaces'' $\bar{X}_j$, $j=1,2$.
    \end{lemma}   
    
    \begin{proof} Suppose $S_b$ is defined by the equation
      \[ f_b(x_0,x_1,x_2,x_3)=0\ .\]
      Since $x_0, x_1$ only occur in even degrees, $S_b$ is the inverse image of a surface $S_b^\prime\subset \PP^\prime:=\PP(2,2,2,5)$ under the $(\ZZ/2\ZZ)^2$ covering map
    \[  \PP \ \to\ \PP^\prime\ .\]    
  The $\ZZ/2\ZZ$ covering
    \[ \PP^\prime\ \to\ \PP(2,2,2,10) \]
    is an isomorphism, since both these spaces are isomorphic to $\PP(1,1,1,5)$ \cite{Dol}. This implies that $S_b$ is isomorphic to a surface in $\PP$ defined by an equation
    \[ x_3^2=g(x_0^2,x_1^2,x_2)\ , \]  
    where $g$ is a quintic. The result now follows from Proposition \ref{pz}.
        \end{proof}

  \begin{lemma}\label{relCK} Let $\Ss\to B$ be the universal family of special Horikawa surfaces (cf. Notation \ref{fam}). There exist relative correspondences
    \[ \pi^0_\Ss\ ,\ \ \pi^2_\Ss\ ,\ \ \pi^4_\Ss\ \ \ \in A^2(\Ss\times_B \Ss)\ ,\]
    with the property that for each $b\in B$, the restriction
    \[ \pi^i_\Ss\vert_b := \pi^i_\Ss\vert_{S_b\times S_b}\ \ \ \in A^2(S_b\times S_b) \]
    is a self-dual Chow--K\"unneth decomposition. Moreover,
     \[ (\pi^2_\Ss\vert_b)_\ast=\ide\colon\ \ \ A^2_{hom}(S_b)\ \to\ A^2_{hom}(S_b)\ \ \ \forall\ b\in B\ ,\]
     and
     \[   \pi^0_\Ss\vert_b = {1\over 2}\, (f_i)^\ast(o_i)\times S_b\ \ \ \hbox{in}\ A^2(S_b\times S_b)\ \ \ \forall\ b\in B\ ,\ i=1,2\ ,\]
     where $f_i\colon S_b\dashrightarrow X_i$ denotes the rational map to the associated K3 surface $X_i$, and $o_i\in A^2(X_i)$ is the distinguished zero-cycle of \cite{BV}.
      \end{lemma}
     
    \begin{proof} The existence of relative correspondences is well-known, and holds more generally for any family of surfaces $\Ss\to B$ with $H^1(S_b)=0$. Let $H\in A^1(\Ss)$ be a relatively ample divisor, and let $d:=\deg (H^2\vert_{S_b})$. One defines
      \[ \begin{split}   \pi^0_\Ss &:= {1\over d} (p_1)^\ast(H^2)\ ,\\
                               \pi^4_\Ss &:= {1\over d} (p_2)^\ast(H^2)\ ,\\
                               \pi^2_\Ss &:= \Delta_\Ss - \pi^0_\Ss -\pi^4_\Ss\ \ \ \ \in\ A^2(\Ss\times_B \Ss)\ .\\
                               \end{split}\]
                   It is readily checked this defines a self-dual Chow--K\"unneth decomposition on any fiber.
                   
    Next, let us restrict to the family $\Ss\to B$ of Notation \ref{fam}. Then, taking $H$ to be the pullback of the hyperplane section of $\PP^2$, the restriction
    $H\vert_{S_b}$ comes from a divisor $\bar{h}_i$ on the ``singular K3 surface'' $\bar{X_i}$ for $i=1,2$, and so
     \[  H^2\vert_{S_b}= (\bar{f}_i)^\ast (\bar{h}_i)^2= (\bar{f}_i)^\ast  (g_i)_\ast (g_i)^\ast   (\bar{h}_i)^2 =(f_i)^\ast(o_i)   \ \ \ \hbox{in}\ A^2(S_b)\ \ \ \forall\ b\in B\ ,\ i=1,2\ ,\]
     as requested.                         
        \end{proof}

 \subsection{Transcendental part of the motive}
 
 \begin{theorem}[Kahn--Murre--Pedrini \cite{KMP}]\label{t2}    Let $S$ be any smooth projective surface, and let $h(X)\in\MM_{\rm rat}$ denote the Chow motive of $S$.
There exists a self-dual Chow--K\"unneth decomposition $\{\pi^i_S\}$ of $S$, with the property that there is a further splitting in orthogonal idempotents
  \[ \pi^2_S= \pi^{2,alg}_{S}+\pi^{2,tr}_{S}\ \ \hbox{in}\ A^2(S\times S)_{}\ .\]
  The action on cohomology is
  \[  (\pi^{2,alg}_{S})_\ast H^\ast(S,\QQ)= N^1 H^2(S,\QQ)\ ,\ \ (\pi^{2,tr}_{S})_\ast H^\ast(S,\QQ) = H^2_{tr}(S,\QQ)\ ,\]
  where the transcendental cohomology $H^2_{tr}(S,\QQ)\subset H^2(S,\QQ)$ is defined as the orthogonal complement of $N^1 H^2(S,\QQ)$ with respect to the intersection pairing. The action on Chow groups is
  \[ (\pi^{2,alg}_{S})_\ast A^\ast(S)_{}= N^1 H^2(S,\QQ)\ ,\ \ (\pi^{2,tr}_{S})_\ast A^\ast(S) = A^2_{AJ}(S)_{}\ .\]  
 This gives rise to a well-defined Chow motive
  \[ h^{tr}_2(S):= (S,\pi^{2,tr}_{S},0)\ \subset \ h(X)\ \ \in\MM_{\rm rat}\ ,\]
  the so-called {\em transcendental part of the motive of $S$}.
  \end{theorem}

\begin{proof} Let $\{\pi^i_S\}$ be a Chow--K\"unneth decomposition as in \cite[Proposition 7.2.1]{KMP}. The assertion then follows from \cite[Proposition 7.2.3]{KMP}.
\end{proof} 
   
 \begin{remark} The construction of Theorem \ref{t2}, using only formal properties, is still valid for surfaces that are quotient varieities.
 \end{remark}

 \section{Results}

 \subsection{Splitting the motive}
 
\begin{theorem}\label{main} Let $S$ be a special Horikawa surface, and let $X_1, X_2$ be the associated $K3$ surfaces. There is an isomorphism
     \[ A^2_{hom}(S)\cong A^2_{hom}(X_1)\oplus A^2_{hom}(X_2)\ .\] 
     Moreover, for an appropriate choice of points defining $h^0(S)$ and $h^0(X_j)$,
     there is an isomorphism of motives
     \[ h^2_{tr}(S)\cong h^2_{tr}(X_1)\oplus h^2_{tr}(X_2)\ \ \ \hbox{in}\ \MM_{\rm rat}\ .\]
     
  \end{theorem}

\begin{proof} This is proven in \cite[Theorem 3.1]{Gsurf} for Garbagnati surfaces (special Horikawa surfaces are birational to Garbagnati surfaces of type G2a in the terminology of loc. cit.). For completeness, we include the argument.

The statement for Chow groups follows from the statement for Chow motives. This last statement is proven 
by exploiting the bidouble cover structure. Let us define motives $h(S)^{\pm \mp}\in\MM_{\rm rat}$ by setting
   \[  \begin{split}
                 h(S)^{++}&:=(S,{1\over 4}(\Delta_S+\Gamma_{\sigma_1})\circ(\Delta_S+\Gamma_{\sigma_2}),0)\ ,\\
                  h(S)^{+-}&:=(S,{1\over 4}(\Delta_S+\Gamma_{\sigma_1})\circ(\Delta_S-\Gamma_{\sigma_2}),0)\ ,\\ 
                  h(S)^{-+}&:=(S,{1\over 4}(\Delta_S-\Gamma_{\sigma_1})\circ(\Delta_S+\Gamma_{\sigma_2}),0)\ ,\\ 
                  h(S)^{--}&:=(S,{1\over 4}(\Delta_S-\Gamma_{\sigma_1})\circ(\Delta_S-\Gamma_{\sigma_2}),0)\ .\\ 
                  \end{split}\]
         (It is readily checked that the given cycles are idempotents and so define motives.)
                  
  This gives a decomposition
    \[ h(S)=h(S)^{++}\oplus h(S)^{+-}\oplus h(S)^{-+}\oplus h(S)^{--}\ \ \ \hbox{in}\ \MM_{\rm rat}\ .\]                
  Defining $h^0(S)$ and $h^4(S)$ by the choice of a zero-cycle invariant under $\sigma_1$ and $\sigma_2$, we get a similar decomposition for $h^2(S)=h(S)-h^0(S)-h^4(S)$. One can check (cf. \cite[Proof of Theorem 3.1]{Gsurf}) that this decomposition is compatible with the decomposition
  $h^2(S)=h^2_{tr}(S)\oplus h^2_{alg}(S)$, and hence there is an induced decomposition
          \[ h^2_{tr}(S)=h^2_{tr}(S)^{++}\oplus h^2_{tr}(S)^{+-}\oplus h^2_{tr}(S)^{-+}\oplus h^2_{tr}(S)^{--}\ \ \ \hbox{in}\ \MM_{\rm rat}\ .\]       
          The first summand $h^2_{tr}(S)^{++}$ is the transcendental part of the motive of $\PP^2$, which is zero. 
          The
          summand $h^2_{tr}(S)^{--}$ corresponds to $h^2_{tr}(W)$ where $W$ is as in the proof of Proposition \ref{Htr}. Since $W$ is a
           rational surface, the summand $h^2_{tr}(S)^{--}$ is zero.          
           The summands $h^2_{tr}(S)^{+-}$ and $h^2_{tr}(S)^{-+}$ are isomorphic to
           $h^2_{tr}(\bar{X}_1)=h^2_{tr}(X_1)$ resp.   $h^2_{tr}(\bar{X}_2)=h^2_{tr}(X_2)$. This proves the theorem.
 \end{proof}

 \begin{corollary}\label{cor1} Let $S$ be a special Horikawa surface, and assume
   \[ \dim H^2_{tr}(S,\QQ)\le 5\ .\]
   Then $S$ has finite-dimensional motive (in the sense of Kimura \cite{Kim}, \cite{An}, \cite{J4}).
   \end{corollary}
   
  \begin{proof} Let $X_1, X_2$ be the associated K3 surfaces. Recall (Proposition \ref{Htr}) that there is an isomorphism
     \[ H^2_{tr}(S,\QQ)\cong H^2_{tr}(X_1,\QQ)\oplus H^2_{tr}(X_2,\QQ)\ .\]  
   The $X_j$ being K3 surfaces, the dimension of $H^2_{tr}(X_j,\QQ)$ is at least $2$, and so the assumption on $ H^2_{tr}(S,\QQ)$ implies that
   \[ \dim  H^2_{tr}(X_j,\QQ)\le 3\ \ \ (j=1,2)\ .\]
   It follows from \cite{Ped} that $X_1$ and $X_2$ have finite-dimensional motive. In view of the isomorphism of Theorem \ref{main}, this implies the corollary.
   \end{proof}

 \subsection{The Franchetta property}

 \begin{lemma}\label{smooth} Let $\Ss\to B$ be as in Notation \ref{fam}. The variety $\Ss$ is a smooth quasi-projective variety.
 \end{lemma}
 
 \begin{proof} By construction, there are morphisms
   \[ \begin{array}[c]{ccc} 
      \Ss & \xrightarrow{}& \PP\\
     \    \downarrow{\scriptstyle } &&\\
      B&&
      \end{array}\]
   Let $\bar{\Ss}\to\bar{B}$ denote the universal family of all (not necessarily smooth) hypersurfaces in $\PP$ of type
    \[ f_b(x_0,x_1,x_2,x_3)=0\ ,\]
    where $f_b$ is weighted homogeneous of degree $10$  and $x_0, x_1$ only occur in even degrees. Then $\bar{B}$ is a projective space containing $B$ as a Zariski open. We make the following claim:
    
    \begin{claim}\label{bpf} For any $x\in\PP$, there exists $b\in\bar{B}$ such that $x\not\in S_b$.
    \end{claim}
    
   \begin{proof} There is a $(\ZZ/2\ZZ)^2$ cover
     \[ \PP\ \to\  \PP^\prime:=\PP(2,2,2,5)\ .\]
     The surfaces in $\bar{\Ss}\to \bar{B}$ correspond to the complete linear system $\PP H^0(\PP^\prime,\OO_{\PP^\prime}(10))$ which is (ample hence) base point free.
    \end{proof}

    Claim \ref{bpf} ensures that $\bar{\Ss}$ is a projective bundle over $\PP$, in particular it is a projective quotient variety. Any surface $S_b$ with $b\in B$ avoids the two singular points of $\PP$, and so $\Ss$ is Zariski open inside a projective bundle over the non-singular locus of $\PP$.
    It follows that $\Ss$ is smooth. 
       \end{proof}
 
 We are now in position to prove that the universal family of special Horikawa surfaces has the Franchetta property:
 
 \begin{proposition}\label{fran} Let $\Ss\to B$ be as in Notation \ref{fam}. Then for any $b\in B$,
   \[ \ima\Bigl( A^2(\Ss)\to A^2(S_b)\Bigr)=\QQ \]
   injects into cohomology under the cycle class map. 
   
   A generator for $ \ima\bigl( A^2(\Ss)\to A^2(S_b)\bigr)$ is 
     \[ o_{S_b}:={1\over 2}(f_1)^\ast(o_1)={1\over 2}(f_2)^\ast(o_2)\ \ \in\ A^2(S_b)\ ,\] 
     where $f_i\colon S_b\dashrightarrow X_{bi}$ ($i=1,2$) are the rational maps to the associated K3 surfaces $X_{bi}$, and $o_i\in A^2(X_{bi})$ is the distinguished zero-cycle of \cite{BV}.
   \end{proposition}
   
   \begin{proof}
    Given a cycle $\alpha\in A^j(\Ss)$, there exists $\bar{\alpha}\in A^j(\bar{\Ss})$ such that $\alpha$ is the
   restriction of $\bar{\alpha}$. Since $p\colon \bar{\Ss}\to\PP$ is a projective bundle, we can write
     \[  \bar{\alpha}= \sum_k p^\ast(a_k) \xi^{j-k}\ \ \ \hbox{in}\ A^j(\bar{\Ss})\ ,\]
     where $\xi\in A^1(\bar{\Ss})$ is a relatively ample class, and $a_k\in A^k(\PP)$.
     
Let $h\in A^1(\bar{B})$ be a hyperplane section, and let $q\colon
		\bar{\Ss} \to \bar{B}$ denote the projection. We have
					$q^\ast(h)=\nu\, \xi+p^\ast(b)$, for some $\nu\in\QQ$ and $b\in A^1(\PP
		)$. It is readily checked that $\nu$ is non-zero. (Indeed, assume for a moment
		$\nu$ were zero. Then we would have $q^\ast(h^{\dim\bar{B}})=p^\ast(b^{\dim
		\bar{B}})$ in $A^{\dim\bar{B}}(\bar{\Ss})$. But the right-hand side is
		zero, since $\dim\bar{B}>\dim \PP=3$, whereas the left-hand side is non-zero;
		contradiction.)		
				The constant $\nu$ being non-zero, we can write
				\[ \xi= p^\ast(b)+q^\ast(c)\ \ \ \hbox{in}\ A^1(\bar{\Ss})\ ,\]
				where $b\in A^1(\PP)$ and $c\in A^1(\bar{B})$ are non-zero elements.
				The restriction of $q^\ast(c)$ to a fiber $V_b$ is zero,
				and so we find that
				\[  \bar{\alpha}\vert_{S_b}=  a_j^\prime\vert_{S_b}\ \ \
		\hbox{in}\
				A^j(S_b)\ ,\]
				for some $a_j^\prime\in A^j(\PP)$. That is, we have proven equality

   \begin{equation}\label{gendef}  \ima\Bigl( A^j(\Ss)\to A^j(S_b)\Bigr)=\ima\Bigl( A^j(\PP)\to  A^j(S_b)\Bigr)\ .      \end{equation}
   Since $A^2(\PP)$ is one-dimensional, this proves the first statement.
   
   For the second statement, we recall that there is a degree 4 morphism $q\colon S_b\to \PP^2$. Letting $h\in A^1(\PP^2)$ denote the hyperplane class, the zero-cycle 
     \[ o_{S_b}:={1\over 4}q^\ast(h^2)\ \ \in\ A^2(S_b)\] 
is in the image of the specialization map $ \ima\bigl( A^2(\Ss)\to A^2(S_b)\bigr)$, and has degree 1. Since $o_{S_b}$ is the unique degree 1 zero-cycle invariant under both involutions $\sigma_i$, the other descriptions follow.
    \end{proof}
 
 \begin{proposition}\label{rho1} The very general special Horikawa surface has Picard number 1.
  \end{proposition}
  
  \begin{proof} (NB: this is observed in \cite[Remark 1.10]{PZ}.)
  The special Horikawa surface defined by the equation
  \[x_3^2=x_0^{10}+x_1^{10}+x_2^5\] 
  in $\PP(1,1,2,5)$ has Picard number $1$, in view of \cite[Proposition 4]{Moo}.
    \end{proof}
  
\subsection{Establishing an MCK}

  \begin{theorem}\label{th:mck} Let $S$ be a special Horikawa surface. Then $S$ has an MCK decomposition. In particular, the image of the intersection product map
  \[ A^1(S)\otimes A^1(S)\ \to\ A^2(S) \]
  is of dimension $1$, generated by $c_2(S)$.
  \end{theorem}
  
  \begin{proof} This proof is inspired by \cite[Proposition 6.1]{FLV}. Let $\Ss\to B$ be the universal family as in Notation \ref{fam}, and let $\pi^j_{\Ss}$ be the relative CK decomposition as in Lemma \ref{relCK}. For any fiber $S:=S_b$ with $b\in B$, the fiberwise restriction $\pi^j_{S}:=\pi^j_{\Ss}\vert_{S\times S}$ defines a CK decomposition; this corresponds to choosing
    \[ o_{S}={1\over 2}(f_i)^\ast(o_i)\ \ \ \in\ A^2(S)\]
 as the distinguished zero-cycle (where $f_i\colon S\dashrightarrow X_i$ ($i=1,2$) are, as before, the rational maps to the associated K3 surfaces). We will show that this CK decomposition is MCK. By a standard spread argument (cf. \cite[Lemma 3.2]{Vo}), it suffices to prove this for $b\in B$ very general, and so we may suppose that $S$ has Picard number $1$ (Proposition \ref{rho1}).
  
  Let us consider the decomposition of motives
    \begin{equation}\label{decompi} h(S)= h(S)^{++}\oplus h(S)^{+-}\oplus h(S)^{-+}\oplus h(S)^{--}\ \ \hbox{in}\ \MM_{\rm rat} \end{equation}
   as in the proof of Theorem \ref{main}.
   By the choice of the zero-cycle $o_S$ defining $\pi^0_{S}$ and $\pi^4_{S}$, we have that $h^0(S)$ and $h^4(S)$ are submotives of
   $h(S)^{++}$. Since $\pic(S)=\QQ$ we moreover have that $h^2_{alg}(S)\subset h(S)^{++}$, and so
    \[  h(S)^{++}= h^0(S)\oplus h^2_{alg}(S)\oplus h^4(S)\cong h(\PP^2)\ \ \ \hbox{in}\ \MM_{\rm rat}\ .\]
    It follows that there is equality
     \[  h(S)^{+-}\oplus h(S)^{-+}\oplus h(S)^{--}= h^2_{tr}(S)\ \ \hbox{in}\ \MM_{\rm rat}\ . \]  
     But then, in view of Theorem \ref{main}, the motive $h(S)^{--}$ must be $0$. The decomposition (\ref{decompi}) thus boils down to
     \begin{equation}\label{decom} \begin{split}   h(S)&= h(S)^{++}\oplus h^2_{}(S)^{+-}\oplus h^2_{}(S)^{-+}\\
                                                                                    &=: (S,\Delta^{++},0)\oplus (S,\Delta^{+-},0)\oplus (S,\Delta^{-+},0)  \ \ \ \hbox{in}\ \MM_{\rm rat} \ .\\
                                                                                    \end{split}\end{equation}
     
 In view of Proposition \ref{equiv}, to show that $\{\pi^j_{S}\}$ is MCK it suffices to prove vanishing of the modified small diagonal
      \[\begin{split} \Gamma_3:=\Gamma_{3}(S,
			o_S):=&\Delta^{sm}_{S}-p_{12}^{*}(\Delta_{S})p_{3}^{*}(o_S)-p_{23}^{*}(\Delta_{S})p_{1}^{*}(o_S)-p_{13}^{*}(\Delta_{S})p_{2}^{*}(o_S)\\
			&+p_{1}^{*}(o_S)p_{2}^{*}(o_S)+p_{1}^{*}(o_S)p_{3}^{*}(o_S)+p_{2}^{*}(o_S)p_{3}^{*}(o_S)  \ \ \ \ \hbox{in}\ A^4(S^3)\ .\\
			\end{split}\]  
 
  Using the decomposition \eqref{decom}, we see that there is equality
    \[  \Gamma_3= \sum_{ A,B,C\in \{ \Delta^{++},\Delta^{+-},\Delta^{-+}\} }       (A\times B\times C)_\ast (\Gamma_3)\ \ \ \hbox{in}\ A^4(S^3)\ .\] 
    Rewriting this using the inclusion-exclusion principle, we find
     \begin{equation}\label{inout} \begin{split} \Gamma_3                                              &= \sum_{ A,B,C\in \{ \Delta^{++},\Delta^{+-}\} }       (A\times B\times C)_\ast (\Gamma_3) +     \sum_{ A,B,C\in \{ \Delta^{++},\Delta^{-+}\} }  (A\times B\times C)_\ast (\Gamma_3)\\
                                                &\ \ \   - (\Delta^{++}  \times\Delta^{++}\times\Delta^{++})_\ast (\Gamma_3) + \sum_{Rest}    (A\times B\times C)_\ast (\Gamma_3)\ \ \ \  \hbox{in}\ A^4(S^3)\ ,\\
                                                \end{split}\end{equation}
                                                where $\displaystyle\sum_{Rest}$ is the sum over all combinations $A,B,C\in\{\Delta^{+-},\Delta^{-+}\}$ such that not all three are the same.
   Let us ascertain that each of the four summands in \eqref{inout} is zero:
   
   \begin{itemize}
   
   \item For the first summand, we note that $\Delta^{++}+\Delta^{+-}$ is a projector on $\ima\bigl( A^\ast(\bar{X}_1)\to A^\ast(S)\bigr)$, and so
     \[  \begin{split} \sum_{ A,B,C\in \{ \Delta^{++},\Delta^{+-}\} }       (A\times B\times C)_\ast (\Gamma_3) &={1\over 8} (\bar{f}_1\times\bar{f}_1\times \bar{f}_1)^\ast (\bar{f}_1\times \bar{f}_1\times \bar{f}_1)_\ast (\Gamma_3)\\   
     &=  {1\over 4} (\bar{f}_1\times\bar{f}_1\times\bar{f}_1)^\ast  \Gamma_3(\bar{X}_1, (\bar{f}_1)_\ast(o_S))\\   
     &=  {1\over 8} (\bar{f}_1\times \bar{f}_1\times \bar{f}_1)^\ast (g_1\times g_1\times g_1)_\ast   \Gamma_3(X_1, o_1))\\     
     &=0\ .\\
     \end{split}\]
     (Here, $\bar{f}_1\colon S\to \bar{X}_1$ and $g_1\colon X_1\to\bar{X}_1$ are as in the proof of Proposition \ref{Htr}.  The last equality expresses the fact that the distinguished zero-cycle $o_1$ defines an MCK decomposition for the K3 surface $X_1$.
    In the other equalities, we have twice used that for any morphism of degree $d$ between surfaces $f\colon S\to X$ and any $a\in A^2(S)$, one has
           \[    (f\times f\times f)_\ast   \Gamma_3(S,a)  =d\, \Gamma_3(X, f_\ast(a))\ \ \ \hbox{in}\ A^4(X^3)\ .)\]
         
     \item For the second summand, the argument is the same, replacing $X_1$ by $X_2$.
     
     \item As for the third summand, $\Delta^{++}$ is a projector on $\ima(A^\ast(\PP^2)\to A^\ast(S))$, and so this summand vanishes by the same argument, replacing $X_i$ by $\PP^2$ and using that $\PP^2$ has an MCK decomposition.
     
     \item Finally, the last summand consists of terms
       \[      (A\times B\times C)_\ast (\Gamma_3)\ ,\ \ \    \   A,B,C\in \{\Delta^{+-},\Delta^{-+}\}\ .\]
     Looking at the definition of $\Gamma_3$, it is direcly checked that any term of this type vanishes, because 
         \[    (A\times B\times C)_\ast (\Delta^{sm}_S)=0\ \ \ \hbox{in}\ A^4(S^3)\ ,\ \ \    \   A,B,C\in \{\Delta^{+-},\Delta^{-+}\}\ \]       
         (indeed, one has equality $(\Delta^{+-}\times \Delta^{+-}\times \Delta_S)_\ast(\Delta^{sm}_S)= (\Delta^{+-}\times \Delta^{+-}\times \Delta^{++})_\ast(\Delta^{sm}_S)$, which  
         expresses the fact that $\Delta^{sm}_S$ is invariant under both involutions $\sigma_i$), and
         \[ (\Delta^{+-})_\ast (o_S)= (\Delta^{-+})_\ast (o_S)=0\ \ \ \hbox{in}\ A^2(S) \]
         (which expresses the fact that $o_S$ is invariant under both involutions $\sigma_i$).
              \end{itemize}

  This proves that the CK decomposition is MCK.
        To finish the proof of the theorem, it only remains to see that $c_2(S)\in A^2(S)$ is in $A^2_{(0)}(S)$, i.e. we need to check that 
        \[ (\pi^j_S)_\ast c_2(S)=0\ \ \ \hbox{in}\ A^2(S)\ \ \forall j\not=4\ .\]
        This follows from the Franchetta property (Proposition \ref{fran}), since we can write
        \[   (\pi^j_S)_\ast c_2(S)=    \Bigl( (\pi^j_{\Ss})_\ast c_2(T_{\Ss/B})\Bigr){}\vert_{S_b}  \ \ \ \hbox{in}\ A^2(S)\  .\]
        (Another way to prove that $c_2(S)$ is proportional to the distinguished zero-cycle on $S$ is by using the ``modified small diagonal relation'', as is done in 
        \cite{BV} for K3 surfaces.)
        Then, $c_2(S)\in   A^2_{(0)}(S)\cong\QQ$ is a generator, because
        $ \deg c_2(S)=    \chi_{top}(S)=   35 $.  
        The theorem is now proven. 
          \end{proof}

  \begin{corollary}\label{cor5} Let $S$ be a special Horikawa surface, and $m\in\NN$. Let $R^\ast(S^m)\subset A^\ast(S^m)$ be the $\QQ$-subalgebra
    \[ R^\ast(S^m):=\langle p_j^\ast A^1(S), (p_{ij})^\ast \Delta_S \rangle\ \ \ \subset\ A^\ast(S^m) \]
    (here, $p_j$ and $p_{ij}$ denote projections from $S^m$ to $S$ resp. to $S\times S$). 
    
    Then $R^j(S^m)$ injects into cohomology under the cycle class map for $j\ge 2m-1$.
    \end{corollary}
    
    \begin{proof} This is an immediate consequence of the MCK package. Indeed, the product $S^m$ has an MCK decomposition \cite[Theorem 8.6]{SV}.  
 One has $A^1(S)=A^1_{(0)}(S)$ (because $\pi^1_S=0$) and $\Delta_S\in A^2_{(0)}(S\times S)$ (this is true for any MCK decomposition, cf. \cite[Lemma 1.4]{SV2}). The projections $p_j$ and $p_{ij}$ respect the grading \cite[Corollary 1.6]{SV2}, and so $R^\ast(S^m)\subset A^\ast_{(0)}(S^m)$.
The corollary now follows from the fact that for any surface $S$ with an MCK decomposition, and any $m\in\NN$, the cycle class map induces injections
  \[ A^i_{(0)}(S^m)\ \hookrightarrow\ H^{2i}(S^m)\ \ \ \forall i\ge 2m-1\ \]
  (this is noted in \cite[Introduction]{V6}, cf. also \cite[Proof of Lemma 2.20]{acs}).        
    \end{proof}

  \begin{remark} Conjecturally, $R^j(S^m)$ should inject into cohomology for all $j$. This is similar to a conjecture made by Voisin for K3 surfaces
  \cite[Conjecture 1.6]{Voi08}.
    \end{remark}

\vskip1cm
\begin{nonumberingt} Thanks to the referee for suggesting the proof of Theorem \ref{th:mck}; this is far more elegant than my original argument.
Thanks to Kai and Len for countless enjoyable "Papa Lunches".
\end{nonumberingt}

\end{document}